\documentclass{amsart}
\usepackage{amscd}
\usepackage{textcomp}
\usepackage{amsfonts}
\usepackage{amsthm}
\usepackage{geometry}
\usepackage{amsmath}
\usepackage{txfonts}
\usepackage{graphicx}
\usepackage{wrapfig}
\usepackage{subcaption}
\usepackage{ascmac}
\usepackage{bm}
\usepackage{indentfirst}
\usepackage{comment}
\usepackage{amssymb}
\usepackage{textcomp}
\usepackage{xcolor}
\usepackage{tikz}
\usetikzlibrary{calc}
\usepackage[normalem]{ulem}
\usepackage{hyperref}

\usepackage[british]{babel}

\providecommand{\keywords}[1]{\vspace{2mm}\noindent\textbf{Keywords:} #1\par}
\newtheorem{theorem}{Theorem}[section]
\newtheorem{proposition}[theorem]{Proposition}
\newtheorem{lemma}[theorem]{Lemma}

\theoremstyle{definition}
\newtheorem{definition}[theorem]{Definition}
\newtheorem{remark}[theorem]{Remark}

\title{Hypercube Embeddings and Median Structure in the Intersection Lattice of Discriminantal Arrangements $\mathcal{B}(n,k)$}

\author{Pragnya Das}

\address{Department of Mathematics, IIIT Kota, India.}

\email{pragnya.hmas@iiitkota.ac.in}

\thanks{}

\begin{document}

% ------------------------------------------------
\begin{abstract}
We investigate the metric structure of the intersection lattice $\mathcal{L}(\mathcal{B}(n,k))$ of the discriminantal arrangement using circuit supports.
We show that the cover graph associated with $\mathcal{L}(\mathcal{B}(n,k))$ is isometrically embedded into a hypercube, making it a partial cube and a median graph, with distances given by the Hamming distance and geodesics described by symmetric differences. We also prove a Poisson limit and a sharp threshold for overlaps of random circuit families, revealing an underlying hypercube geometry.
\end{abstract}

\maketitle

% ------------------------------------------------
\section{Introduction}

Hyperplane arrangements and their intersection lattices play a central role in combinatorics, geometry, and topology. The intersection lattice encodes important combinatorial information about the arrangement and is closely related to algebraic and topological invariants. (See \cite{OrlikTerao} for reference).\\
\noindent
Among the important families of arrangements is the discriminantal arrangement $\mathcal{B}(n,k)$ introduced by Manin and Schechtman~\cite{ManinSchechtman}. Discriminantal arrangement is an arrangement of hyperplanes arising from degeneracy conditions among parallel translates of the $n$ hyperplanes in general position in $\mathbb{R}^k$, generalizing the classical braid arrangement. Its combinatorial structure is closely connected with higher Bruhat orders and with concepts from oriented matroids theory. \cite{ManinSchechtman,BjornerOM}\\
\noindent
A basic property of the discriminantal arrangement is that its hyperplanes are indexed by $(k+1)$-subsets of $[n]$(see, e.g., \cite{ManinSchechtman, OrlikTerao}). These subsets correspond to minimal degeneracy conditions and will be referred to as circuits of the arrangement. As a consequence, the set of circuits naturally induces the structure of the Johnson graph $J(n,k+1)$.\\
\noindent
In this paper we study graph structures naturally associated with the discriminantal arrangement. Two graphs arise in this setting. The first is the Johnson graph formed by the circuits of the arrangement. The second is the cover graph of the intersection lattice $\mathcal{L}(\mathcal{B}(n,k))$, whose vertices correspond to nonempty intersections of hyperplanes and whose edges represent covering relations in the lattice.\\
\noindent
While the Johnson graph is a classical object in algebraic graph theory and forms an important family of distance-regular graphs (see \cite{GodsilRoyle2001,BrouwerCohenNeumaier}), the cover graph of the intersection lattice exhibits a different behaviour. It is naturally graded by rank and reflects a collection of degeneracy conditions among circuit hyperplanes. Understanding the metric structure of this graph provides insight into how local combinatorial constraints organize the global geometry of the lattice.\\
\noindent
We begin by examining the circuit structure of $\mathcal{B}(n,k)$, which induces the Johnson graph $J(n,k+1)$ and encodes the interaction of individual circuits. Building on this viewpoint, we then pass to general lattice elements via their circuit supports and develop a metric description of the cover graph of $\mathcal{L}(\mathcal{B}(n,k))$. We show that distances between lattice elements are governed by symmetric differences of feasible supports, building an isometric embedding into a hypercube. As a consequence, the cover graph is a partial cube and, a median graph.\\
\noindent
We further describe the structure of shortest paths by introducing a dependency poset that determines the admissible sequences of circuit modifications, and show that geodesics correspond to linear extensions of this poset. We also establish that every interval in the lattice is isomorphic to a hypercube, revealing a strong local cubical structure.\\
\noindent
Finally, we investigate random circuit supports. We prove that overlaps between random supports admit a Poisson approximation in the sparse regime and exhibit a sharp threshold for intersection. As a consequence, distances between random lattice elements concentrate near their maximal value, reflecting the high-dimensional geometry of the underlying hypercube embedding.\\
\noindent
The paper is organized as follows. Section~\ref{prelims} reviews basic properties of the discriminantal arrangement and its intersection lattice. Section~\ref{MainR} briefly introduced all the results of the paper. Section~\ref{Johnson} studies the Johnson graph on circuits. Section~\ref{Metric} establishes the metric structure via circuit supports and proves the hypercube embedding and median property. Section~\ref{Convexity} analyzes convexity and cube structure of intervals. Section~\ref{Asymp} investigates the asymptotic behaviour of random supports.\\
\noindent
These results show that the intersection lattice of the discriminantal arrangement admits a unified cubical geometry, in which both local and global metric properties are governed by simple combinatorial interactions between circuit supports.

\section{Main Results}\label{MainR}
We summarize the principal structural and metric results established in this paper.\\
\noindent
Let $C_{n,k} = \binom{[n]}{k+1}$ denote the set of circuits, and for $X \in L(B(n,k))$, let $F(X)$ denote its (feasible) circuit support.

\begin{itemize}

\item \textbf{Metric structure.} 
We show that the distance between lattice elements is governed by symmetric differences of their supports:
\[
d(X,Y) = |F(X) \triangle F(Y)|.
\]
In particular, the cover graph of $L(B(n,k))$ admits an isometric embedding into the hypercube $\{0,1\}^{C_{n,k}}$.

\item \textbf{Cubical and median structure.}
We prove that the cover graph is a partial cube and, moreover, a median graph. The median of three lattice elements is given by a coordinate-wise majority rule on their supports.

\item \textbf{Geodesic structure.}
We characterize shortest paths between lattice elements in terms of admissible sequences of circuit modifications. These are governed by a dependency poset, and geodesics are in bijection with its linear extensions.

\item \textbf{Interval structure.}
We show that every interval $[X,Y]$ in $L(B(n,k))$ is isomorphic to a hypercube. In particular, intervals inherit a Boolean lattice structure determined by subsets of $F(Y)\setminus F(X)$.

\item \textbf{Asymptotic behaviour.}
We study random circuit supports and prove that overlaps admit a Poisson approximation in the sparse regime, together with a sharp threshold for intersection. As a consequence, distances between random lattice elements concentrate near their maximal value.

\end{itemize}
\noindent
Taken together, these results show that the intersection lattice of the discriminantal arrangement admits a unified cubical geometry, in which both local and global metric properties are governed by simple combinatorial interactions between circuit supports.
%----------------%-----------------------------------
\section{Preliminaries}\label{prelims}
We recall basic definitions and notation related to the discriminantal arrangement and its intersection lattice.
\subsection{The discriminantal arrangement}

Consider $n$ hyperplanes $\{H^0_i\}_{i \in [n]=\{1,...,n\}}$, where $n$ and $k$ be integers with $n \ge k+1$, in general position in $\mathbb{R}^{k}$ and all of their parallel translates be denoted by the set $\mathbb{S}$. The \emph{discriminantal arrangement}, denoted by $\mathcal{B}(n,k)$, describes the degeneracy conditions that occur when these translated hyperplanes become affinely dependent, i,e., $$\mathcal{B}(n, k)=\{D_I \subset  \mathbb{S} \mid I \subset [n], | I |=k+1\},$$ where the hyperplanes corresponding to the indices in $I$ are affinely dependent.\\
\noindent
A fundamental property of the discriminantal arrangement is that its hyperplanes are naturally indexed by subsets of the index set $[n] = \{1,2,\dots,n\}$ \cite{ManinSchechtman}. Geometrically, the hyperplane $D_I\in\mathcal{B}(n,k)$ represents the condition that the $k+1$ translated hyperplanes indexed by the elements of $I$ become affinely dependent. Since affine dependence among $k+1$ hyperplanes in $\mathbb{R}^k$ is a minimal degeneracy condition, the subsets $I$ become the circuits in the combinatorial structure of the arrangement. Consequently, the set of circuits of $\mathcal{B}(n,k)$ is in one-to-one correspondence with the collection of all $(k+1)$-subsets of $[n]$, and the set of all such circuits be denoted by $C_{n,k} = \binom{[n]}{k+1}$ through out this paper. 
\subsection{Intersection lattice}
The combinatorial structure of a hyperplane arrangement is encoded by its intersection lattice. For the discriminantal arrangement $\mathcal{B}(n,k)$ we denote this lattice by $\mathcal{L}(\mathcal{B}(n,k))$.\\
\noindent
The elements of $\mathcal{L}(\mathcal{B}(n,k))$ are all nonempty intersections of hyperplanes of the form $D_I$, where $I\in C_{n,k}$. Thus each lattice element can be written in the form
\[
X = \bigcap_{I \in \mathcal{F}} D_I
\]
for some family $\mathcal{F} \subseteq C_{n,k}$.\\
\noindent
The lattice is partially ordered by reverse inclusion. In other words, for two lattice elements $X$ and $Y$ we define
\[
X \le Y \quad \text{if and only if} \quad Y \subseteq X.
\]
\noindent
Under this ordering the minimal element $\hat{0}$ corresponds to the entire ambient space \cite{OrlikTerao}.

\subsection{Minimal defining families and Circuit support}

A lattice element $X \in \mathcal{L}(\mathcal{B}(n,k))$ can have several representations as an intersection of circuit hyperplanes. One of them is the minimal defining family of circuits. 
\begin{definition}
 A family of circuits $\{I_1,\dots,I_r\}$ is called a \emph{minimal defining family} for $X$ if
\[
X = \bigcap_{i=1}^{r} D_{I_i}
\]
and removing any circuit $I_j$ strictly enlarges the intersection.
\end{definition}
\noindent
Minimal defining families describe collections of non-redundant degeneracy conditions and provide a useful local description of lattice elements. While they provide one possible representation of lattice elements, they lack the ability to describe the global structural properties. In this paper, we instead use the canonical support representation defined as follows:
\noindent
For $X \in \mathcal{L}(\mathcal{B}(n,k))$, we define
\[
F(X) = \{ I \in C_{n,k} : X \subseteq D_I \}.
\]
\noindent
This representation is canonical and uniquely determines $X$, since $X = \bigcap_{I \in F(X)} D_I.$
\begin{definition}[Feasible support]
A family of circuits $F \subseteq C_{n,k}$ is called \emph{feasible} if the intersection
\[
X_F := \bigcap_{I \in F} D_I
\]
is nonempty.
\end{definition}

\begin{remark}
Only feasible families correspond to elements of the intersection lattice $L(B(n,k))$. In particular, every lattice element $X \in L(B(n,k))$ is uniquely represented by its support
\[
F(X) = \{ I \in C_{n,k} : X \subseteq D_I \},
\]
and satisfies
\[
X = \bigcap_{I \in F(X)} D_I \neq \emptyset.
\]
\noindent
Conversely, an arbitrary subset $F \subseteq C_{n,k}$ does not necessarily define a lattice element unless the associated intersection is nonempty. 
\end{remark}
\noindent
We will implicitly restrict all support families to feasible ones throughout the paper.
\subsection{Rank structure}

The intersection lattice $\mathcal{L}(\mathcal{B}(n,k))$ is a graded poset, where the rank of $X\in \mathcal{L}(\mathcal{B}(n,k))$ is defined by
\[
\operatorname{rk}(X) = \operatorname{codim}(X).
\]
Thus the rank records the number of independent degeneracy conditions defining $X$.
\\
The lattice decomposes into rank levels
\[
L(B(n,k)) = \bigsqcup_{r \ge 0} L_r, \quad 
L_r = \{ X \in L(B(n,k)) : \operatorname{rk}(X) = r \}.
\]
\noindent
While this grading provides a natural layering, the metric structure studied in this paper is governed primarily by circuit supports rather than rank.\\
 In the following sections, we use support representation $F(X)$, which provides a natural way to represent the lattice elements as subsets of $C_{n,k}$, to study the cover graph of $\mathcal{L}(\mathcal{B}(n,k))$. We further show that the cover graph admits an isometric embedding into a hypercube and exhibits a rich metric structure.

%%%%%%%%%%%%%%%%%%%%%%%%%%%%%%%%%%%%%%%%%%%%%%%%%%%%%%%%%%-----------------------------------------------------------------------------------Jonshon Graph------------------------------------%%%%%%%%%%%%%%%%%%%%%%%%%%%%%%%%%%%%%%%%%%%%%%%%%-----

\section{The Johnson graph on circuits}\label{Johnson}

As described in the previous section, the circuits of the discriminantal arrangement $\mathcal{B}(n,k)$ are indexed by the $(k+1)$ subsets of $[n]$. In this section, we study the graph structure that is naturally induced by the collection these circuits. 

\subsection{Definition of the circuit graph}

\begin{definition}
The \emph{circuit graph} of the discriminantal arrangement $\mathcal{B}(n,k)$ is the graph whose vertex set is $C_{n,k}$, where two vertices $I$ and $J$ are adjacent if and only if
\[
|I \cap J| = k .
\]
\end{definition}
\noindent
Equivalently, two circuits are adjacent when one can be obtained from the other by replacing exactly one element as each circuit contains $k+1$ indices. 

\begin{proposition}
The circuit graph defined above is isomorphic to the Johnson graph $J(n,k+1)$.
\end{proposition}

\begin{proof}
By definition, the Johnson graph $J(n,r)$ has vertex set consisting of the $r$ subsets of $[n]$, where two vertices are adjacent whenever their intersection has size $r-1$. Taking $r = k+1$ yields the adjacency condition
\[
|I \cap J| = k ,
\]
which coincides with the adjacency relation defined above. Hence the circuit graph is precisely the Johnson graph $J(n,k+1)$.
\end{proof}
\noindent
Next, we list all the basic properties of the circuit structure of the discriminantal arrangement inherited by the combinatorial structure of Johnson graphs which is a well studied object in algebraic graph theory.
\subsection{Basic properties}

\begin{proposition}

The Johnson graph $J(n,k+1)$ has the following properties
(see e.g. \cite[Ch.~9]{GodsilRoyle2001}, \cite{BrouwerCohenNeumaier}) and so does the circuit structure of the discriminantal arrangement:
\begin{enumerate}
    \item The graph is regular of degree $(k+1)(n-k-1)$.
    \item The symmetric group $S_n$ acts transitively on vertices, hence the graph is vertex–transitive.
    \item The distance between vertices $I$ and $J$ satisfies
\[
d(I,J)=(k+1)-|I\cap J|.
\]
\item The diameter equals
\[
\min\{k+1,n-k-1\}.
\]
\end{enumerate}    
\end{proposition}
\subsection{Geometric interpretation}

The Johnson graph also admits a natural polyhedral interpretation. It appears as the $1$ skeleton of the hypersimplex
\[
\Delta(n,k+1) = \{x \in [0,1]^n : x_1 + \cdots + x_n = k+1\}.
\]
\noindent
The vertices of this polytope correspond to $0$-$1$ vectors with exactly $k+1$ ones, which are in bijection with $(k+1)$ subsets of $[n]$. Two vertices are adjacent when one coordinate equal to $1$ is replaced by a coordinate equal to $0$, which corresponds precisely to the adjacency relation of the Johnson graph. This geometric interpretation highlights the strong symmetry present in the circuit structure of the discriminantal arrangement.\\

\subsection{Connection to the metric structure of the lattice}

The Johnson graph structure on circuits provides a local model for the metric behaviour of the intersection lattice.\\
\noindent
As mentioned above, vertices of the Johnson graph $J(n,k+1)$ correspond to individual circuits, and the distance between two circuits $I$ and $J$ is given by$d(I,J) = (k+1) - |I \cap J|.$
Thus distances are governed by the size of overlaps, or equivalently, by symmetric differences of subsets.\\
\noindent
This observation extends naturally to general lattice elements when they are represented by their circuit supports. Indeed, for $X, Y \in L(B(n,k))$, the supports $F(X)$ and $F(Y)$ are subsets of $C_{n,k}$, and the distance in the cover graph will be shown to satisfy
\[
d(X,Y) = |F(X) \triangle F(Y)|.
\]
\noindent
In this sense, the Johnson graph captures the metric behaviour at the level of single circuits, while the cover graph of the intersection lattice arises as a higher-dimensional extension in which vertices correspond to families of circuits rather than individual ones.\\
\noindent
Moreover, the hypercube embedding of the cover graph can be viewed as a natural generalization of the hypersimplex realization of the Johnson graph: while circuits correspond to vertices of a fixed-weight slice of the hypercube, general lattice elements correspond to arbitrary feasible subsets of circuits. This transition from fixed-size subsets to general supports leads to the emergence of a cubical geometry.\\
\noindent
Thus, the Johnson graph provides the fundamental combinatorial building block for the global metric structure of $L(B(n,k))$, and serves as a guiding model for the hypercube geometry developed in the subsequent sections.\\
\noindent
In the next section, we enhance this viewpoint from individual circuits to general lattice elements by representing each element through its circuit support. This leads to a global description of the cover graph of $\mathcal{L}({B}(n,k))$ as an isometric subgraph of a hypercube.

%%%%%%%%%%%%%%%%%%%%%%%%%%%%%%%%%%%%%%%%%%%%%%%%%%%%%%%%%%%--------------------------------------------------------------------------INTERSECTION GRAPH -----------------------------------%%%%%%%%%%%%%%%%%%%%%%%%%%%%%%%%%%%%%%%%%%%%%%%%%%%%%%%%%%%%%%%%
\section{Metric Structure via Circuit Supports}\label{Metric}

The intersection lattice $\mathcal{L}(\mathcal{B}(n,k))$ inherits a rich combinatorial structure, 
but its metric properties are not immediately clear. In particular, 
distances in the cover graph depend on how degeneracy conditions accumulate, 
and it is not clear a priori how to measure the distance between two lattice elements.\\
\noindent
A naive approach based on minimal defining families is insufficient, since such 
representations are not unique. To overcome this, we use the canonical representation using circuit supports, which allows us to encode each lattice element in a unique and natural way.

\subsection{Canonical support representation and Local moves}

For $X \in L(\mathcal{B}(n,k))$, The canonical support representation,$F(X)$ records all circuit hyperplanes containing $X$ and provides a canonical 
description of the degeneracy conditions satisfied by $X$.\\
\noindent
Recall that we restricted all support families to feasible ones throughout the paper.
\begin{proposition}
For every $X \in L(\mathcal{B}(n,k))$, we have
\[
X = \bigcap_{I \in F(X)} D_I.
\]
In particular, the map $X \mapsto F(X)$ is injective.
\end{proposition}

\begin{proof}
By definition, $F(X)$ consists of all circuit hyperplanes containing $X$, 
so clearly
\[
X \subseteq \bigcap_{I \in F(X)} D_I.
\]
\noindent
Conversely, suppose $x$ lies in every $D_I$ with $X \subseteq D_I$. Since $X$ is itself the intersection of all hyperplanes defining it, any point satisfying 
all these hyperplanes must lie in $X$. Hence equality holds.\\
If $F(X)=F(Y)$, then both $X$ and $Y$ are intersections of the same family of 
hyperplanes, so $X=Y$.
\end{proof}
\noindent
A natural way to understand the structure of the lattice $\mathcal{L}(\mathcal{B}(n,k))$ is to interpret each of its intersections combinatorially. So, instead of working directly with intersections of hyperplanes, we record which circuit hyperplanes are present.\\
To formalize this, we associate to every element $X\in\mathcal{L}(\mathcal{B}(n,k))$ to its support. The support representation allows us to encode lattice elements as binary vectors as follows;\\
Define
\[
\varphi : L(\mathcal{B}(n,k)) \to \{0,1\}^{C_{n,k}}, \quad 
\varphi(X) = \chi_{F(X)}.
\]
\noindent
Each coordinate corresponds to a circuit, and records whether that circuit 
contains $X$.\\
\noindent
In this way, every lattice element is viewed as a point in a high-dimensional hypercube, and the combinatorics of the intersection lattice begins to resemble that of binary vectors.\\
\noindent
Next, we study the distances in the cover graph as it allows us to pass from the combinatorial definition of the lattice to an explicit geometric model. To understand distances, we first analyze cover relations.\\
\begin{lemma}\label{lem:circuit}
If $X$ and $Y$ are adjacent in the cover graph, then their supports differ 
by exactly one circuit, i.e.,
\[
|F(X)\triangle F(Y)| = 1.
\]
\end{lemma}

\begin{proof}
A cover relation corresponds to changing the rank by one, which means adding 
or removing a single independent degeneracy condition.\\
\noindent
In terms of supports, this corresponds to adding or removing exactly one 
circuit hyperplane. No other circuits can change, since that would either 
leave the rank unchanged or change it by more than one.
\end{proof}
\noindent
Lemma~\ref{lem:circuit} suggests that we can move through the lattice by modifying supports one circuit at a time. Now the question arises, whether such modifications always correspond to valid lattice elements.

\begin{lemma}\label{lem:non-emp}
Let $X \in L(\mathcal{B}(n,k))$ and let $I \in C_{n,k}$. Then the families 
$F(X) \cup \{I\}$ and $F(X) \setminus \{I\}$ correspond to lattice 
elements whenever they define nonempty intersections of hyperplanes.
\end{lemma}

\begin{proof}
Recall that every subset of circuits corresponds to an intersection of the 
associated hyperplanes:
\[
X_F = \bigcap_{J \in F} D_J.
\]
\noindent
Thus, any feasible family $F \subseteq C_{n,k}$ defines a lattice element.\\
Now consider $F(X)$.\\
\noindent \textbf{(1) Adding a circuit.}
The family $F(X) \cup \{I\}$ corresponds to the intersection
\[
X' = X \cap D_I.
\]
Since $X$ is nonempty and $D_I$ is a hyperplane, this intersection is either 
equal to $X$ or has codimension increased by one. In particular, it defines 
a valid lattice element.

\medskip

\noindent \textbf{(2) Removing a circuit.}
The family $F(X) \setminus \{I\}$ corresponds to removing one constraint, 
which enlarges the intersection:
\[
X'' = \bigcap_{J \in F(X)\setminus\{I\}} D_J.
\]
This is again a nonempty intersection and hence a valid lattice element.\\
Thus adding or removing a single circuit yields a valid lattice element whenever the resulting support remains feasible.
\end{proof}

\subsection{Metric Structure and Isometric Embedding}

With the local structure understood, we now turn to a global description of distances.

\begin{theorem}[Metric Representation]\label{thm:Metric}
For all $X,Y \in L(\mathcal{B}(n,k))$, the distance in the cover graph satisfies
\[
d(X,Y) = |F(X)\triangle F(Y)|.
\]
\end{theorem}

\begin{proof}
We prove both inequalities.\\
Each step in any path changes exactly one circuit in the support. Therefore, any path from $X$ to $Y$ must modify each element of $F(X)\triangle F(Y)$ at least once. Hence
\[
d(X,Y) \ge |F(X)\triangle F(Y)|.
\]
\noindent
We now construct a path of this length.\\
Let
\[
A = F(X)\setminus F(Y), \quad B = F(Y)\setminus F(X).
\]
\noindent
We first remove circuits in $A$ one by one, and then add circuits in $B$.\\
By Lemma~\ref{lem:non-emp} and feasibility of supports, each intermediate family defines a lattice element.\\
This produces a path of length $|A|+|B| = |F(X)\triangle F(Y)|$.\\
Since both bounds coincide, equality holds.
\end{proof}
\noindent
The distance formula shows that the support representation is not just convenient, but exact.
\begin{theorem}\label{thm:emb}
The map $\varphi$ is an isometric embedding. In particular, the cover graph 
of $\mathcal{L}(\mathcal{B}(n,k))$ is a partial cube.
\end{theorem}

\begin{proof}
By the theorem~\ref{thm:Metric}, the distance between $X$ and $Y$ equals the Hamming 
distance between $\varphi(X)$ and $\varphi(Y)$. Hence $\varphi$ preserves 
distances and defines an isometric embedding into the hypercube.
\end{proof}
\begin{theorem}[Median graph property]\label{thm:Median}
The cover graph of $L(B(n,k))$ is a median graph. That is, for any three vertices $X, Y, Z \in L(B(n,k))$, there exists a unique vertex $M$ such that
\[
d(X,Y) = d(X,M) + d(M,Y), \quad
d(X,Z) = d(X,M) + d(M,Z), \quad
d(Y,Z) = d(Y,M) + d(M,Z).
\]
\end{theorem}

\begin{proof}
By Theorem~\ref{thm:emb}, the map
\[
\phi : L(B(n,k)) \to \{0,1\}^{C_{n,k}}, \quad \phi(X) = \chi_{F(X)},
\]
is an isometric embedding, where distances correspond to Hamming distance.\\
\noindent
Let $x = \phi(X)$, $y = \phi(Y)$, and $z = \phi(Z)$. In the hypercube $\{0,1\}^{C_{n,k}}$, the unique median is given coordinate-wise by
\[
m_I = \mathrm{maj}(x_I, y_I, z_I), \quad I \in C_{n,k}.
\]
\noindent
Define
\[
F(M) = \{ I \in C_{n,k} : m_I = 1 \}
= (F(X) \cap F(Y)) \cup (F(Y) \cap F(Z)) \cup (F(Z) \cap F(X)).
\]
\noindent
We show that $F(M)$ is feasible.\\
\noindent
Let $I \in F(M)$. Then $I$ belongs to at least two of the supports $F(X), F(Y), F(Z)$. Without loss of generality, suppose $I \in F(X) \cap F(Y)$. Then both $X$ and $Y$ lie in the hyperplane $D_I$, and hence their intersection $X \cap Y$ is contained in $D_I$.\\
\noindent
Thus every hyperplane indexed by $F(X) \cap F(Y)$ contains the nonempty set $X \cap Y$. The same holds for $F(Y) \cap F(Z)$ and $F(Z) \cap F(X)$.\\
\noindent
Now consider the family $F(M)$. Each constraint in $F(M)$ is satisfied by at least one of the nonempty sets
\[
X \cap Y, \quad Y \cap Z, \quad Z \cap X.
\]
\noindent
Since $X, Y, Z$ are intersections of hyperplanes, these sets are affine subspaces, and any two among them are consistent intersections of hyperplanes. It follows that the collection of constraints indexed by $F(M)$ is jointly consistent, and hence
\[
\bigcap_{I \in F(M)} D_I \neq \emptyset.
\]
\noindent
Thus $F(M)$ defines a lattice element $M$.\\
By construction, $\phi(M) = m$. Since medians in the hypercube are unique and $\phi$ is an isometric embedding, $M$ is the unique vertex lying on shortest paths between each pair among $X, Y, Z$. This proves the result.
\end{proof}
\subsection{Geodesic structure}

We now describe shortest paths in the cover graph.
\begin{definition}[Dependency poset]
Let $X, Y \in L(B(n,k))$ and define
\[
A = F(X) \setminus F(Y), \quad 
B = F(Y) \setminus F(X).
\]
\noindent
Consider the set
\[
S = A \cup B.
\]
\noindent
We define a partial order $\preceq$ on $S$ as follows:\\
\noindent
For $I, J \in S$, we declare $I \preceq J$ if, in every sequence of single-circuit modifications transforming $F(X)$ into $F(Y)$ through feasible supports, the modification corresponding to $I$ must occur before that of $J$.\\
\noindent
Equivalently, $I \preceq J$ if there exists no feasible sequence in which $J$ is performed before $I$.
\end{definition}
\begin{theorem}[Geodesic characterization]\label{thm:geo}
Let $X, Y \in L(B(n,k))$. Then every geodesic between $X$ and $Y$ is obtained by modifying the circuits in $F(X) \triangle F(Y)$ exactly once, in an order that respects the dependency poset $(S, \preceq)$.\\
In particular, geodesics are in bijection with the linear extensions of $(S, \preceq)$.
\end{theorem}

\begin{proof}
By the distance formula, we have
\[
d(X,Y) = |F(X) \triangle F(Y)| = |S|.
\]
\noindent
Thus any geodesic must modify each circuit in $S$ exactly once.\\
Each step in a path corresponds to adding or removing a single circuit, and the intermediate supports must remain feasible. Therefore, a sequence of modifications defines a valid path if and only if every intermediate support corresponds to a feasible family.\\
By definition of the dependency poset, a modification corresponding to $I$ must precede that of $J$ whenever performing $J$ before $I$ would violate feasibility. Hence any valid sequence of modifications must respect the partial order $\preceq$.\\
Conversely, any linear extension of $(S, \preceq)$ defines an ordering of the circuits that respects all feasibility constraints, and thus yields a valid path from $X$ to $Y$. Since each circuit is modified exactly once, such a path has length $|S|$ and is therefore a geodesic.\\
This establishes a bijection between geodesics and linear extensions of the dependency poset.
\end{proof}

\begin{remark}
If all circuits in $S$ are independent in the sense that no feasibility constraints arise, then the dependency poset is trivial, and the number of geodesics is $|S|!$.
\end{remark}
\noindent
This reflects the fact that geodesics in the cover graph behave analogously 
to shortest paths in a hypercube, where each coordinate is modified exactly once.

\section{Convexity and Cube Structure of Intervals}\label{Convexity}

The hypercube embedding obtained in the previous section suggests that the metric structure of $\mathcal{L}(\mathcal{B}(n,k))$ may admit a stronger geometric description. While we now understand distances globally, it is natural to ask whether this cubical behavior persists locally, inside intervals of the lattice.\\
\noindent 
We show that this is indeed the case: every interval inherits a canonical cube structure.\\

\subsection{Interval Structure and Convexity}

Let $X,Y \in L(\mathcal{B}(n,k))$ with $X \le Y$. Then $F(X) \subseteq F(Y).$\\
\noindent
Thus every element $Z$ in the interval $[X,Y]$ satisfies $F(X) \subseteq F(Z) \subseteq F(Y).$ Thus, elements in the interval correspond precisely to supports lying between two fixed sets. We now show that this description is compatible with the metric structure.

\begin{theorem}
For any $X,Y \in L(\mathcal{B}(n,k))$, the interval $[X,Y]$ is convex in the cover graph.
\end{theorem}

\begin{proof}
Let $Z_1,Z_2 \in [X,Y]$. Then
\[
F(X) \subseteq F(Z_1),F(Z_2) \subseteq F(Y).
\]
\noindent
Any shortest path between $Z_1$ and $Z_2$ modifies circuits in $F(Z_1)\triangle F(Z_2)$ one at a time. Since both supports lie between $F(X)$ and $F(Y)$, all intermediate supports also lie between them. Hence every vertex along the path remains in $[X,Y]$.
\end{proof}

\subsection{Cube Structure of Intervals}
We now strengthen the previous result by giving an explicit description of intervals.
\begin{theorem}
For any $X,Y \in L(\mathcal{B}(n,k))$, the interval $[X,Y]$ is isomorphic to the 
hypercube of dimension $|F(Y)\setminus F(X)|$.
\end{theorem}

\begin{proof}
Let $S = F(Y)\setminus F(X).$ Every element $Z \in [X,Y]$ is obtained by choosing a subset $T \subseteq S$ 
and defining
\[
F(Z) = F(X) \cup T.
\]
\noindent
This gives a bijection between $[X,Y]$ and $2^S$. Two elements are adjacent if their supports differ by one circuit, which corresponds to adjacency in the hypercube. Thus $[X,Y]$ is isomorphic to $Q_{|S|}$.
\end{proof}
\noindent
We illustrate the cube structure of intervals in Figure~\ref{fig:cube}.
\begin{figure}[h]
\centering
\begin{tikzpicture}[scale=1.2, every node/.style={circle, draw, minimum size=6mm}]

% Bottom square
\node (000) at (0,0) {$F(X)$};
\node (100) at (2,0) {$F(X)\cup\{a\}$};
\node (010) at (0,2) {$F(X)\cup\{b\}$};
\node (110) at (2,2) {$F(X)\cup\{a,b\}$};

% Top square
\node (001) at (0.8,0.8) {$F(X)\cup\{c\}$};
\node (101) at (2.8,0.8) {$F(X)\cup\{a,c\}$};
\node (011) at (0.8,2.8) {$F(X)\cup\{b,c\}$};
\node (111) at (2.8,2.8) {$F(Y)$};

% Bottom edges
\draw (000)--(100);
\draw (000)--(010);
\draw (100)--(110);
\draw (010)--(110);

% Top edges
\draw (001)--(101);
\draw (001)--(011);
\draw (101)--(111);
\draw (011)--(111);

% Vertical edges
\draw (000)--(001);
\draw (100)--(101);
\draw (010)--(011);
\draw (110)--(111);

\end{tikzpicture}

\caption{An interval $[X,Y]$ in $L(B(n,k))$ represented as a 3-dimensional cube. 
Each vertex corresponds to a feasible support of the form $F(X)\cup T$, where $T \subseteq \{a,b,c\} = F(Y)\setminus F(X)$. 
Edges correspond to adding a single circuit, and distances agree with Hamming distance.}\label{fig:cube}
\end{figure}
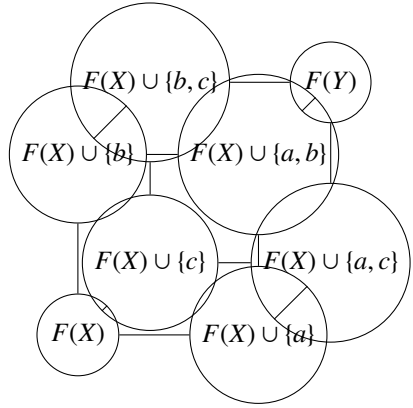
This result shows that the cubical structure observed at the global level does in fact persists locally. Not only is the cover graph a partial cube, but every interval decomposes as a canonical hypercube, whose dimension is determined by the number of independent circuit extensions between its endpoints. In particular, the lattice admits a regular metric geometry, where both global distances and local substructures are governed by the same combinatorial principle.

%%%---------Section 5----------------%%%%%%%%%%%%%%%%%%

\section{Asymptotic Behaviour of Circuit Supports}\label{Asymp}

The geometric description developed in the previous sections shows that the cover graph of $\mathcal{L}(\mathcal{B}(n,k))$ admits a cubical structure, in which distances are 
determined by symmetric differences of circuit supports. In particular, the distance between two lattice elements $X$ and $Y$ depends on the size of the overlap $|F(X)\cap F(Y)|$.\\
\noindent
This naturally leads to the following question: what is the typical size of the overlap between two circuit supports when the number of circuits becomes large?\\
\noindent
In this section we address this question in a probabilistic setting. We show that, in the sparse regime, overlaps between random circuit families are typically negligible, leading to distances that concentrate near their maximal possible value. This phenomenon reflects the high-dimensional hypercube geometry underlying the lattice.\\
\noindent
We now study the typical behaviour of overlaps between random circuit supports and its implications for the metric structure of the cover graph. Let $N = \binom{n}{k+1}$ denote the total number of circuits. Let $F$ and $G$ be independent uniformly chosen subsets of $C_{n,k}$ of size $r$. Define
\[
T = |F \cap G|.
\]
\noindent
The quantity $T$ measures the overlap between the supports and directly influences distances in the lattice. We begin by analyzing its distribution.

\begin{theorem}[Poisson approximation with error bound]
The random variable $T$ follows a hypergeometric distribution with parameters $(N, r, r)$.\\
\noindent
If $r = o(\sqrt{N})$, then $T$ converges in distribution to a Poisson random variable with parameter
\[
\lambda = \frac{r^2}{N}.
\]
\noindent
Moreover, there exists an absolute constant $C > 0$ such that for every integer $t \geq 0$,
\[
\left| \mathbb{P}(T = t) - e^{-\lambda} \frac{\lambda^t}{t!} \right|
\leq C \frac{r^3}{N^2}.
\]
\end{theorem}

\begin{proof}
Since $F$ and $G$ are chosen uniformly with size $r$, the size of their intersection $T$ follows a hypergeometric distribution with parameters $(N, r, r)$.\\
\noindent
When $r = o(\sqrt{N})$, the dependence between the underlying indicator variables becomes negligible. By standard results on Poisson approximation for sampling without replacement (see \cite{BarbourHolstJanson1992}), the total variation distance between $T$ and a Poisson random variable with parameter $\lambda = r^2/N$ is bounded by a constant multiple of $r^3/N^2$.\\
\noindent
This bound tends to zero as $n \to \infty$, proving convergence in distribution.
\end{proof}
\noindent
The Poisson approximation shows that overlaps are typically small in the sparse regime. We now refine this by identifying when intersections occur with high probability.

\begin{theorem}[Sharp threshold for intersection]\label{thm:thres}
Let $r = r(n)$ and $N = \binom{n}{k+1}$. Then:
\begin{itemize}
\item If $r = o(\sqrt{N})$, then $\mathbb{P}(F \cap G \neq \emptyset) \to 0$,
\item If $r = \omega(\sqrt{N})$, then $\mathbb{P}(F \cap G \neq \emptyset) \to 1$.
\end{itemize}
\end{theorem}

\begin{proof}
We compute
\[
\mathbb{P}(F \cap G = \emptyset) = \frac{\binom{N - r}{r}}{\binom{N}{r}}.
\]

This can be written as
\[
\mathbb{P}(F \cap G = \emptyset)
= \prod_{i=0}^{r-1} \left(1 - \frac{r}{N - i}\right).
\]

Taking logarithms and using $\log(1 - x) = -x + O(x^2)$, we obtain
\[
\log \mathbb{P}(F \cap G = \emptyset)
= -\frac{r^2}{N} + O\left(\frac{r^3}{N^2}\right).
\]

Exponentiating gives
\[
\mathbb{P}(F \cap G = \emptyset)
= \exp\left(-\frac{r^2}{N} + O\left(\frac{r^3}{N^2}\right)\right).
\]

If $r = o(\sqrt{N})$, then $r^2/N \to 0$, so $\mathbb{P}(F \cap G = \emptyset) \to 1$.  
If $r = \omega(\sqrt{N})$, then $r^2/N \to \infty$, so $\mathbb{P}(F \cap G = \emptyset) \to 0$.
\end{proof}
\noindent
We now translate these probabilistic results into geometric information about the lattice.

\begin{theorem}[Metric concentration]
Let $X_F$ and $X_G$ be lattice elements corresponding to supports $F$ and $G$ of size $r$. Then
\[
d(X_F, X_G) = 2r - 2|F \cap G|.
\]
\noindent
Moreover, if $r = o(\sqrt{N})$, then
\[
\mathbb{P}(d(X_F, X_G) = 2r) \to 1.
\]
\end{theorem}

\begin{proof}
By the symmetric difference identity,
\[
d(X_F, X_G) = |F \triangle G| = |F| + |G| - 2|F \cap G| = 2r - 2|F \cap G|.
\]

From the theorem~\ref{thm:thres}, when $r = o(\sqrt{N})$, we have $\mathbb{P}(|F \cap G| = 0) \to 1$. Substituting into the distance formula gives the result.
\end{proof}
\noindent
Taken together, these results show that in the sparse regime, random supports have negligible overlap and therefore correspond to lattice elements that are at near-maximal distance. This provides a probabilistic explanation for the high-dimensional cubical geometry observed in the cover graph.\\
\noindent
Thus, the asymptotic behaviour of circuit supports provides further evidence 
that the intersection lattice of the discriminantal arrangement exhibits a 
robust cubical geometry, both locally (through intervals) and globally (through metric concentration).
\section{Conclusion}

In this paper, we developed a metric perspective on the intersection lattice $\mathcal{L}(\mathcal{B}(n,k))$ of the discriminantal arrangement. By encoding lattice elements through their circuit supports, we showed that the cover graph admits a canonical cubical geometry: distances are governed by symmetric differences, forming an isometric embedding into a hypercube. This structure persists locally, as every interval is itself a cube, revealing a strong and uniform form of combinatorial convexity. The asymptotic analysis further indicates that, in the sparse regime, overlaps between random circuit families are negligible, so typical elements lie near maximal distance. Together, these results show that both the local and global geometry of the lattice are governed by a single underlying cubical principle. This establishes $\mathcal{L}(\mathcal{B}(n,k))$ as a natural class of median graphs arising from hyperplane arrangements.\\
\noindent Future directions naturally arise from this viewpoint. It would be of interest to refine the metric description further, particularly in understanding extremal and typical distances beyond the sparse regime, and to investigate whether stronger geometric properties emerge from this cubical structure. More broadly, an important question is whether similar metric and probabilistic phenomena persist for intersection lattices of other classes of hyperplane arrangements, suggesting a more general geometric framework underlying these combinatorial objects.
% ------------------------------------------------

\end{document}